\documentclass{amsart}
\usepackage{amsmath,amssymb}
\usepackage[all]{xy}
\usepackage{tabularx}
\usepackage[noend]{algpseudocode}
\usepackage{algorithmicx,algorithm}
\usepackage{threeparttable}
\usepackage{booktabs}

\newtheorem{theorem}{Theorem}

\newtheorem{proposition}{Proposition}[section]

\newtheorem{remark}{Remark}[section]

\numberwithin{equation}{section}

\begin{document}

\title{Supersingular $j$-invariants and the Class Number of $\mathbb{Q}(\sqrt{-p})$}

\author{Guanju Xiao}
\address{Key Laboratory of Mathematics Mechanization, NCMIS, Academy of Mathematics and Systems Science, Chinese Academy of Sciences, Beijing 100190, People's Republic of China}
\address{School of Mathematical Sciences, University of Chinese Academy of Sciences, Beijing 100049, People's Republic of China}
\email{gjXiao@amss.ac.cn}

\author{Lixia Luo}
\address{Key Laboratory of Mathematics Mechanization, NCMIS, Academy of Mathematics and Systems Science, Chinese Academy of Sciences, Beijing 100190, People's Republic of China}
\address{School of Mathematical Sciences, University of Chinese Academy of Sciences, Beijing 100049, People's Republic of China}
\email{luolixia@amss.ac.cn}
\author{Yingpu Deng}
\address{Key Laboratory of Mathematics Mechanization, NCMIS, Academy of Mathematics and Systems Science, Chinese Academy of Sciences, Beijing 100190, People's Republic of China}
\address{School of Mathematical Sciences, University of Chinese Academy of Sciences, Beijing 100049, People's Republic of China}
\email{dengyp@amss.ac.cn}





\begin{abstract}
For a prime $p>3$, let $D$ be the discriminant of an imaginary quadratic order with $|D|< \frac{4}{\sqrt{3}}\sqrt{p}$. We research the solutions of the class polynomial $H_D(X)$ mod $p$ in $\mathbb{F}_p$ if $D$ is not a quadratic residue in $\mathbb{F}_p$. We also discuss the common roots of different class polynomials in $\mathbb{F}_p$. As a result, we get a deterministic algorithm (Algorithm 3) for computing the class number of $\mathbb{Q}(\sqrt{-p})$. The time complexity of Algorithm 3 is $O(p^{3/4+\epsilon})$.
\end{abstract}

\maketitle

\section{Introduction}
For a prime $p>3$, we know that $\mathbb{Q}(\sqrt{-p})$ is an imaginary quadratic field with discriminant $-p$ (resp. $-4p$)  if $p \equiv 3 \pmod 4$ (resp. $p \equiv 1 \pmod 4$). We shall be concerned with the computational problem of calculating the class number of $\mathbb{Q}(\sqrt{-p})$.\par
Currently, the best available rigorous methods for computing class number of an imaginary quadratic field of discriminants $D$ is of complexity $(|D|^{1/2+\epsilon})$(see \cite{MR1933052} and \cite[Proposition 9.7.15]{MR2300780}).
Assuming the Generalized Riemann Hypothesis (GRH), Shanks's algorithm \cite{MR0316385,MR0371855} computes the class number $h(D)$ in $O(|D|^{1/5+\epsilon})$ operations, and can be used to compute the structure of the class group $C(D)$ in $O(|D|^{1/4+\epsilon})$ operations. Moreover, the values of $h(D)$ computed by Shanks's algorithm could not be guaranteed to be correct if the GRH is false. In 1989, Hafner and McCurley \cite{MR1002631} proposed a Las Vegas algorithm that compute the structure of the class group of $C(D)$ in an expected time of $L(D)^{\sqrt{2}+o(1)}$ bit operations under the assumption of GRH, where
$$L(D)=\exp(\sqrt{\log |D| \log \log |D|}).$$\par
In this paper, we will propose algorithms to compute the class number of $\mathbb{Q}(\sqrt{-p})$ using the supersingular $j$-invariants in $\mathbb{F}_p$.\par
For a prime $p>3$, we know that the number of supersingular $j$-invariants in $\mathbb{F}_p$ depends on the class number of $\mathbb{Q}(\sqrt{-p})$. If the discriminant $D$ of an imaginary quadratic order $O$ is not a quadratic residue in $\mathbb{F}_p$, then Deuring's theorem \cite[Theorem 13.12]{MR0409362} shows that the roots of $H_D(X)$ mod $p$ are supersingular $j$-invariants where $H_D(X)$ is the class polynomial. Moreover, Kaneko's theorem \cite[Theorem 1]{MR1040429} shows that every supersingular $j$-invariant contained in the prime field $\mathbb{F}_p$ is a roots of some $H_D(X)$ mod $p$ with $|D| < \frac{4}{\sqrt{3}}\sqrt{p}$. We can get the class number of $\mathbb{Q}(\sqrt{-p})$ by calculating the roots of $H_D(X)$ mod $p$ in $\mathbb{F}_p$ for all $|D|< \frac{4\sqrt{p}}{\sqrt{3}}$ with $\left ( \frac{D}{p} \right ) =-1$, however there exists no efficient algorithm to compute the  class polynomials.\par
To overcome this disadvantage, we consider the prime factorization of $pO_H$ in $H=\mathbb{Q}(j(O))$ where $O$ is an imaginary quadratic order with discriminant $D$. We will show that every prime ideal with norm $p$ in $H$ corresponds to a root of $H_D(X)$ mod $p$ in $\mathbb{F}_p$ if $|D| < \frac{4}{\sqrt{3}}\sqrt{p}$. The new problem is that it is difficult to construct the field $H$ without $j(O)$ in general, so we need a simpler subfield of $H$. Let $F$ be the genus field of $O$ and $L$ be the ring class field of $O$, it's easy to show that the maximal real subfield $E$ of $F$ is a subfield of $H$. What's more, we have the following diagram about the field extensions.
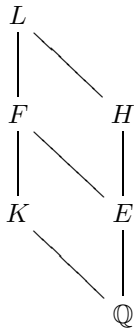
\[ \xymatrix{
&L \ar@{-}[d] \ar@{-}[rd] &\\
&F \ar@{-}[d] \ar@{-}[rd] &H \ar@{-}[d] \\
&K \ar@{-}[rd] &E \ar@{-}[d] \\
&& \mathbb{Q}  } \] \par
It is easier to compute the prime factorization of $p$ in $E$, and we will prove that every prime ideal with norm $p$ in $E$ corresponds to a root of $H_D(X)$ mod $p$ in $\mathbb{F}_p$ with $|D|< \frac{4\sqrt{p}}{\sqrt{3}}$. We also discuss the common roots of different class polynomials $H_{D_1}(X)$ and $H_{D_2}(X)$ in $\mathbb{F}_p$. We propose two algorithms (Algorithm 2 and Algorithm 3) in Section 4. Compared with this algorithm, the theoretical results are more interesting. \par
In this paper, $\text{N}(\mathfrak{p})$ is the absolute norm of $\mathfrak{p}$ if $\mathfrak{p}$ is an ideal in a number field, and $\text{Nrd}(\alpha)$ is the reduced norm of $\alpha$ if $\alpha$ is an element of a quaternion algebra. We always assume $p>3$. \par
The remainder of this paper is organized as follows. In Section 2, we review some preliminaries on supersingular $j$-invariants, class polynomials and genus theory. Theoretical results are in Section 3, including Theorem 5, 8 and 9. We propose two algorithms (Algorithm 2 and Algorithm 3) for computing the class number of $\mathbb{Q}(\sqrt{-p})$ and analyse the time complexity of them in Section 4. Finally, we make a conclusion in Section 5. \par

\section{Preliminaries}

\subsection{Supersingular $j$-invariants}
We will present some basic facts about the supersingular elliptic curves, and the reader can refer to \cite{MR2514094} for more details.
For an elliptic curve $E:Y^2=X^3+aX+b$ over a finite field with characteristic $p>3$, the $j$-invariant of $E$ is $j(E)=1728\cdot 4a^3/(4a^3+27b^2)$. Different elliptic curves with the same $j$-invariant are isomorphic over the algebraic closed field $\overline{\mathbb{F}}_p$. Moreover, the $j$-invariant of every supersingular elliptic curve over $\overline{\mathbb{F}}_p$ is proved to be in $\mathbb{F}_{p^2}$ and it is called a supersingular $j$-invariant. For a supersingular elliptic curve $E$ over $\mathbb{F}_{p^2}$, the endomorphism ring $\text{End}(E)$ is isomorphic to a maximal order of $B_{p,\infty}$, where $B_{p,\infty}$ is a quaternion algebra defined over $\mathbb{Q}$ and ramified at $p$ and $\infty$. Moreover, $j(E)\in \mathbb{F}_p$ if and only if $\text{End}(E)$ contains a root of $x^2+p=0$ (see \cite{MR3451433}). \par
Choose a prime integer $q$ such that $q \equiv 3 \pmod 8$ and $\left( \frac{p}{q} \right) =-1$, then $B_{p,\infty}$ can be written as $B_{p,\infty}=\mathbb{Q}+\mathbb{Q}\alpha+\mathbb{Q}\beta+\mathbb{Q}\alpha\beta$ where $\alpha^2=-p$, $\beta^2=-q$ and $\alpha\beta=-\beta\alpha$. Choosing an integer $r$ such that $r^2+p \equiv 0 \pmod  q$, put
$$\mathcal{O}(q,r)=\mathbb{Z} + \mathbb{Z}\frac{1+\beta}{2} + \mathbb{Z} \frac{\alpha(1+\beta)}{2} + \mathbb{Z}\frac{(r+\alpha)\beta}{q}.$$
When $p \equiv 3 \pmod  4$, we further choose an integer $r'$ such that $r'^2+p \equiv 0 \pmod {4q}$ and put $$\mathcal{O}'(q,r')=\mathbb{Z} + \mathbb{Z}\frac{1+\alpha}{2} + \mathbb{Z} \beta + \mathbb{Z}\frac{(r'+\alpha)\beta}{2q}.$$\par
Ibukiyama's results \cite{MR683249} show that both $\mathcal{O}(q,r)$ and $\mathcal{O}'(q,r')$ are maximal orders of $B_{p,\infty}$ and the endomorphism ring $\text{End}(E)$ is isomorphic to $\mathcal{O}(q,r)$ or $\mathcal{O'}(q,r')$ with suitable choice of $q$ if $E$ is a supersingular elliptic curve over $\mathbb{F}_p$.\par
For a prime $p > 3$, let $S_p$ be the set of all supersingular $j$-invariants in $\mathbb{F}_p$. Then (see \cite{MR3451433})
$$ \# S_{p}= \left\{
\begin{array}{lcl}
\frac{1}{2}h(-4p) & &{\text{if} \ p\equiv 1 \pmod{4} ,} \\
h(-p) & &{\text{if} \ p\equiv 7 \pmod{8} ,}\\
2h(-p) & &{\text{if} \ p\equiv 3 \pmod{8} ,}
\end{array}
\right.
$$
where $h(D)$ is the class number of the imaginary quadratic order with discriminant $D$.

\subsection{Class Polynomials}
We first recall some basic facts about the class polynomials, and the general references are \cite{MR3236783,MR0409362}.
Given an order $O$ in an imaginary quadratic field $K$, the class polynomial $H_D(X)$ is the monic minimal polynomial of $j(O)$ over $\mathbb{Q}$ where $D$ is the discriminant of $O$. Note that $H_D(X)$ has integer coefficients. Let $\{ \mathfrak{a}_i \}$ be a complete set of coset representatives of the $h(D)$ proper ideal classes of $O$. We have that
$$H_D(X)=\prod^{h(D)}_{i=1} (X-j(\mathfrak{a}_i))$$
is an irreducible polynomial in $\mathbb{Z}[X]$. The coefficients of $H_D(X)$ grow rapidly with the size of the discriminant $D$. Sutherland \cite{MR2728992} presented a space-efficient algorithm to compute the class polynomial $H_D(X)$ modulo a positive integer $p$, based on an explicit form of the Chinese Remainder Theorem. Under the Generalized Riemann Hypothesis, the algorithm uses $O(|D|^{1/2+\epsilon} \log p)$ space and has an expected running time of $O(|D|^{1+\epsilon})$. \par
As we known \cite[Lemma 9.3]{MR3236783}, $L=K(j(O))$ is an abelian extension of $K$ and the Galois group $\text{Gal}(L/K)$ is isomorphic to the class group $C(D)$ of $O$. Moreover, $L/\mathbb{Q}$ is a generalized dihedral extension and the Galois group $\text{Gal}(L/\mathbb{Q})$ can be written as a semidirect product
$$\text{Gal}(L/ \mathbb{Q})\simeq \text{Gal}(L/K) \rtimes (\mathbb{Z}/2\mathbb{Z})\simeq C(D)\rtimes (\mathbb{Z}/2\mathbb{Z}),$$
and the nontrivial element of $\mathbb{Z}/2\mathbb{Z}$ acts on $\text{Gal}(L/K)$ via conjugation by $\sigma$ which is a complex conjugation.\par
We also need Deuring's reduction theorem \cite[Theorem 13.12]{MR0409362}.
\begin{theorem}
Let $\tilde{E}$ be an elliptic curve over a number field with $\text{End} (\tilde{E}) \simeq O$, where
$O$ is an order of an imaginary quadratic field $K$. Let $\mathfrak{p}$ be a prime ideal of $\overline{\mathbb{Q}}$ over a prime
number $p$, at which $\tilde{E}$ has non-degenerate reduction $E$. $E$ is supersingular if and only if
$p$ does not split in $K$.
\end{theorem}
Kaneko proved two theorems in \cite{MR1040429}.
\begin{theorem}\label{t2}
  Every supersingular $j$-invariant contained in the prime field $\mathbb{F}_p$ is a root of some $H_D(X)$ mod $p$ with $| D | \le \frac{4}{\sqrt{3}} \sqrt{p}$.
\end{theorem}
\begin{theorem}\label{t3}
If two different discriminants $D_1$ and $D_2$ satisfy $D_1 D_2 < 4p$ (in particular $|D_1|, |D_2| < 2\sqrt{p}$), then two polynomials $H_{D_1}(X)$ mod $p$ and $H_{D_2}(X)$ mod $p$ in $\mathbb{F}_p[X]$ have no roots in common. In other words, every prime factor $p$ of the resultant of $H_{D_1}(X)$ and $H_{D_2}(X)$
satisfies $p \le \frac{D_1 D_2}{4}$.\par
Furthermore, if $\mathbb{Q}(\sqrt{D_1})= \mathbb{Q}(\sqrt{D_2})$, the above inequality $D_1 D_2<4p$ (resp. $p \le \frac{D_1 D_2}{4}$)
can be replaced by $D_1 D_2<p^2$ (resp. $p\le \sqrt{D_1 D_2}$).
\end{theorem}

\subsection{Genus Fields}
Let $K$ be an imaginary quadratic field. $O$ is an order of $K$ with discriminant $D$. We write $D=-4n$ if $D\equiv \ 0 \pmod  4$. The following proposition \cite[Proposition 3.11]{MR3236783} determines the number of elements of order $2$ in the class group $C(D)$.
\begin{proposition}\label{p1}
  Let $D\equiv 0,1 \pmod  4$ be negative, and let $t$ be the number of odd primes dividing $D$. Define the number $\mu$ as follows: if $D\equiv 1 \ \pmod 4$, then $\mu =t$, and if $D\equiv \ 0 \pmod 4$, then $D=-4n$, where $n>0$, and $\mu$ is determined by the following table:
  \begin{table}[H]
    \centering
  \begin{tabular}{c|c}
    \hline
    $n$ & $\mu$ \\
    \hline
    $n \equiv \ 3 \pmod  4$ & $t$ \\
    $n \equiv \ 1,\ 2 \pmod  4$ & $t+1$ \\
    $n \equiv \ 4 \pmod 8$ & $t+1$ \\
    $n \equiv \ 0 \pmod  8$ & $t+2$ \\
    \hline
  \end{tabular}
  \end{table}
  Then the class group $C(D)$ has exactly $2^{\mu-1}$ elements of order $\le 2$.
\end{proposition}
Let $p_1, \ldots, p_t$ be the odd prime divisors of $D$, then the genus field of $O$ \cite[Theorem 2.2.23]{MR1313719} is $K_0(\sqrt{p_1^*}, \ldots, \sqrt{p_t^*})$ where $p_i^*=(-1)^{\frac{p_i-1}{2}}p_i$ and $K_0$ satisfies the following conditions
$$\left \{
\begin{array}{ll}
  K_0=\mathbb{Q} & \text{if} \ D \ \text{odd or} \ n \equiv -1\pmod 4,  \\
  K_0=\mathbb{Q}(i) & \text{if} \ n \equiv  1,4,5\pmod 8, \\
  K_0=\mathbb{Q}(\sqrt{-2}) & \text{if} \ n \equiv 2\pmod 8, \\
  K_0=\mathbb{Q}(\sqrt{2}) & \text{if} \ n \equiv -2\pmod 8, \\
  K_0=\mathbb{Q}(i,\sqrt{2}) & \text{if} \ n \equiv 0\pmod 8.
\end{array}
\right .$$

\section{The Number of Supersingular $j$-invariants over $\mathbb{F}_p$ }
In the first subsection, we will compute the class number of $\mathbb{Q}(\sqrt{-p})$ with class polynomials. The algorithm is simple, however there exists no efficient algorithm to compute class polynomials. In the second subsection, We will research the solutions of the class polynomial $H_D(X)$ mod $p$ in $\mathbb{F}_p$ if $D$ is not a quadratic residue in $\mathbb{F}_p$. We also discuss the common roots of different class polynomials in $\mathbb{F}_p$.
\subsection{Computing the Class Number with Class Polynomials}
Let $D\equiv 0,1 \pmod 4$ be the discriminant of an imaginary quadratic order. For a prime integer $p>3$, if the Legendre symbol $\left (\frac{D}{p} \right ) \neq 1$, then the roots of $H_D(X)$ mod $p$ in $\mathbb{F}_p$ are the supersingular $j$-invariants in $\mathbb{F}_p$ by Deuring's reducing theorem. To get the class number of $\mathbb{Q}(\sqrt{-p})$, we will compute the number of supersingular $j$-invariants over $\mathbb{F}_p$. Theorem \ref{t2} implies that we just need compute all the roots of $H_D(X)$ mod $p$ in $\mathbb{F}_p$ for $|D| < \frac{4}{\sqrt{3}}\sqrt{p}$, so we have Algorithm 1.\par
In fact, the degree of $\text{gcd}(H_D(X),X^p-X)$ is the number of roots of $H_D(X)$ in $\mathbb{F}_p$. Since two class polynomials $H_{D_1}(X)$ mod $p$ and $H_{D_2}(X)$ mod $p$ may have common roots in $\mathbb{F}_p$, we compute explicit roots of $H_D(X)$ in $\mathbb{F}_p$ in Algorithm 1. We will get a sufficient and necessary condition under which we can determine whether $H_{D_1}(X)$ mod $p$ and $H_{D_2}(X)$ mod $p$ have common roots in $\mathbb{F}_p$ or not in the next subsection.\par
The correctness of Algorithm 1 is obvious, but the time complexity is worse than other algorithms. Even though, we can get all the supersingular $j$-invariants over $\mathbb{F}_p$ by Algorithm 1.

\begin{algorithm}[H]
\caption{ Computing the class number of $\mathbb{Q}(\sqrt{-p})$ with $H_D(X)$.}
\label{alg:Framwork}
\begin{algorithmic}[1]
\Require
The prime $p$;
\Ensure
The class number of $\mathbb{Q}(\sqrt{-p})$;
\label{code:fram:extract}
\State Let $S_p=\varnothing$;
\label{code:fram:trainbase}
\State For every negative discriminant $D\equiv 0$ or $1 \pmod 4$ satisfying $|D| < \frac{4}{\sqrt{3}}\sqrt{p}$, do step 3;
\label{code:fram:add}
\State If Legendre symbol $\left (\frac{D}{p} \right )=-1$, then add all the roots of $H_D(X)$ in $\mathbb{F}_p$ to $S_p$;
\label{code:fram:classify}
\State $ h= \left\{
\begin{array}{ll}
\frac{1}{2}\# S_{p} &{\text{if} \ p\equiv 3 \pmod{8} ,} \\
\# S_{p} &{\text{if} \ p\equiv 7 \pmod{8} ,}\\
2\# S_{p} &{\text{if} \ p\equiv 1 \pmod{4} ;}
\end{array}
\right.$
\label{code:fram:select} \\
\Return $h$;
\end{algorithmic}
\end{algorithm}
\subsection{Computing the Number of Supersingular $j$-invariants over $\mathbb{F}_p$ without Class Polynomials}
In this subsection, we will calculate the roots of $H_D(X)$ mod $p$ in $\mathbb{F}_p$ by the prime factorization of $(p)$ in the field $E$ firstly. Secondly, we will consider the common roots of different class polynomials in $\mathbb{F}_p$. By this way, we can compute the number of supersingular $j$-invariants over $\mathbb{F}_p$ without class polynomials.
\subsubsection{The Roots of $H_D(X)$ mod $p$ in $\mathbb{F}_p$}
Let $K$ be an imaginary quadratic field. $O$ is an order of $K$ with discriminant $D$. Let $L$ be the ring class field of $O$ and $H=\mathbb{Q}(j(O))$. We have the following diagram about field extensions.
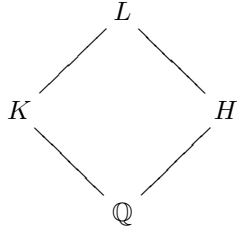
\[ \xymatrix{
&&L \ar@{-}[ld] \ar@{-}[rd] &\\
&K \ar@{-}[rd] &&H \ar@{-}[ld] \\
&&\mathbb{Q} & } \] \par
In general, we have the following theorem \cite[Theorem 3.3.5]{MR1313719} about the prime factorization of prime integer $p$ in number fields.
\begin{theorem}
  Let $H=\mathbb{Q}(\theta)$ be an extension field not necessarily normal but generated by an algebraic integer that satisfies the monic equation
  $$f(x)=x^m+a_1x^{m-1}+ \cdots +a_m=0.$$
  Then $\mathbb{Z}[\theta]$ is generally only a submodule of $O_H$, the integer ring of $H$, of index $i=[O_H:\mathbb{Z}[\theta]]$. Assume $p \nmid i$. Then if we factor
  $$f(x)\equiv f_1(x)^{e_1} \cdots f_g(x)^{e_g} \ \text{mod} \ p$$
  into polynomials $\{f_t(x)\}$ irreducible mod $p$ and distinct, this factorization simulates the prime factorization of $pO_H$,
  $$pO_H=\mathfrak{p}_1^{e_1} \cdots \mathfrak{p}_g^{e_g}.$$
  If the polynomial $f_t(x)$ is of degree $f_t$, then this is the degree of $\mathfrak{p}_t$, which can be written in the form $\mathfrak{p}_t=(p, f_t(\theta))$.
\end{theorem}
In our case, $H=\mathbb{Q}(j(O))$. We assume $\left ( \frac{D}{p} \right ) =-1$ and $|D| < \frac{4}{\sqrt{3}}\sqrt{p}$, so $p$ is inert in $K$ and $pO_K$ splits completely in $L$ which implies that $p$ is not ramified in $H$. We claim that $p$ does not divide the discriminant of $H_D(X)$. If not, there is a multiple root of $H_D(X)$ mod $p$ which implies that $O_1=O$ and $O_2=O$ can be embedded in a maximal order $\mathcal{O}$ of $B_{p,\infty}$ with different images. We have $p\le |D|$ by the proof of Theorem \ref{t3} in \cite{MR1040429}, which contradicts $p>3$ and $p\ge \frac{3D^2}{16}$. We must have $p \nmid [O_H : \mathbb{Z}[j(O)]]$ and the above theorem holds in our case.\par
By Theorem 4, every root of $H_D(X)$ mod $p$ in $\mathbb{F}_p$ corresponds to an ideal $\mathfrak{P}$ in $O_H$ with norm $p$. In other words, we transform the problem of computing the number of roots of $H_D(X)$ mod $p$ in $\mathbb{F}_p$ into the problem of computing the number of integral ideals with norm $p$ in $H$.\par
How can we get the prime factorization of $p$ in $H$ without $H_D(X)$? If the genus field $F$ of $O$ is the same as the ring class field $L$ of $O$,  then $H=L \cap \mathbb{R}$ is a composition of some real quadratic fields. In this case, $H$ is Galois over $\mathbb{Q}$, and it is easy to determine whether $p$ splits completely in $H$ or not. This special case enlightens us to consider the prime factorization of $p$ in the maximal real subfield of $F$. \par
Let $p_1, \ldots, p_t$ be the odd prime divisors of $D$. Suppose that the first $m$ ($m \le t$) primes satisfying $p_i=p_i^*\equiv \ 1  \pmod 4$. For the other odd primes, we define $q_j=p_j^*$ for $j=1,\ldots,n=t-m$. We write $D=-4n$ if $D\equiv \ 0 \pmod  4$.\par
 We define a field $E$ which is a subfield of $F$ as follows.
$$E= \left \{
\begin{array}{ll}
  \mathbb{Q}(\sqrt{p_1},\ldots, \sqrt{p_m}, \sqrt{\frac{D}{q_1}}, \ldots, \sqrt{\frac{D}{q_n}}) & \text{if} \ D \ \text{odd or} \ n \equiv -1 \pmod 4,  \\
  \mathbb{Q}(\sqrt{p_1},\ldots, \sqrt{p_t}) & \text{if} \ n \equiv  1,4,5 \pmod 8, \\
  \mathbb{Q}(\sqrt{p_1},\ldots, \sqrt{p_m}, \sqrt{-2q_1},\ldots, \sqrt{-2q_n}) & \text{if} \ n \equiv 2 \pmod 8, \\
  \mathbb{Q}(\sqrt{2},\sqrt{p_1},\ldots, \sqrt{p_m}, \sqrt{\frac{D}{q_1}}, \ldots, \sqrt{\frac{D}{q_n}}) & \text{if} \ n \equiv -2 \pmod 8, \\
  \mathbb{Q}(\sqrt{2},\sqrt{p_1},\ldots, \sqrt{p_t}) & \text{if} \ n \equiv 0 \pmod 8.
\end{array}
\right .$$
\par
In fact, $E$ is the maximal real subfield of $F$. Moreover, $E$ is a Galois extension over $\mathbb{Q}$ and $\text{Gal}(E/\mathbb{Q})\simeq (\mathbb{Z}/2\mathbb{Z})^{\mu-1}$ where $\mu$ is defined as Proposition \ref{p1}. It is easy to show that $E=F\cap H$. Furthermore, we have the following diagram about the field extensions.
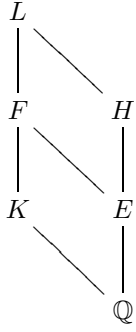
\[ \xymatrix{
&L \ar@{-}[d] \ar@{-}[rd] &\\
&F \ar@{-}[d] \ar@{-}[rd] &H \ar@{-}[d] \\
&K \ar@{-}[rd] &E \ar@{-}[d] \\
&& \mathbb{Q}  } \] \par
As we know, $L/E$ is a Galois extension with Galois group $G'= \text{Gal}(L/E)$. Let $\text{Gal}(L/H)=\langle \sigma \rangle$. If there exists an ideal $\mathfrak{P}$ in $H$ with norm $p$, then $\mathfrak{P}$ is inert in $L$ and the Frobenius map $\left (\frac{\mathfrak{P}}{L/H} \right )=\sigma$ has degree $2$. Moreover, the norm of the prime ideal $\mathfrak{p}=\mathfrak{P}\cap E$ is $p$.

\begin{theorem}
Let $p>3$ be a prime integer and $D$ the discriminant of an imaginary quadratic order. Suppose that $|D|< \frac{4}{\sqrt{3}}\sqrt{p}$ and $\left( \frac{D}{p} \right )=-1$. If $p$ splits completely in $E$, then there exist $2^{\mu-1}$ roots of $H_D(X)$ mod $p$ in $\mathbb{F}_p$. Otherwise, there exists no root of $H_D(X)$ mod $p$ in $\mathbb{F}_p$.
\end{theorem}
\begin{proof}
As we know, the class polynomial $H_D(X)$ mod $p$ has roots in $\mathbb{F}_p$ if and only if there exist prime ideals in $H$ with norm $p$. Let $\mathfrak{P}$ be a prime ideal in $H$ with norm $p$, then the norm of $\mathfrak{p}=\mathfrak{P}\cap E$ is $p$ and $p$ splits completely in $E$. By the proof of Tchebotarev Density Theorem in Janusz's book \cite[Theorem 10.4 Chapter 5]{MR0366864}, the number of prime ideals with norm $p$ in $H$ is $d=[C_{G}(\sigma):\langle \sigma \rangle]$ with $G=\text{Gal}(L/\mathbb{Q})$. For any $\iota \in G$, we have $\iota \sigma \iota^{-1}= \sigma$ if and only if $\sigma \iota \sigma^{-1}=\iota$ if and only if $\iota^{-1}=\iota$ since $\sigma \iota \sigma^{-1}=\iota^{-1}$, so the order of $\iota$ is less than or equal to $2$. As we know, the number of elements with order less than or equal to $2$ in $G$ is $2^\mu$ where $\mu$ is defined in Proposition 2.1, so $d=2^{\mu-1}$. Moreover, we have that $p$ splits into $2^{\mu-1}$ prime ideals in $E$.\par
Suppose that $p$ splits completely in $E$. If $\mathcal{P}$ is a prime ideal in $L$ with norm $p^2$, then the Frobenius map $\left ( \frac{\mathcal{P}}{L/\mathbb{Q}}\right) |_E$ is trivial. Notice that $G'=\text{Gal}(L/E)=C(D)^2 \rtimes \langle \sigma \rangle$, so we have $\left( \frac{\mathcal{P}}{L/\mathbb{Q}} \right) \in G'$. As we know, the order of $\left( \frac{\mathcal{P}}{L/\mathbb{Q}}\right)$ is $2$ and $\left( \frac{\mathcal{P}}{L/\mathbb{Q}} \right)$ is not the square of an element in $C(D)$ since $\left ( \frac{\mathcal{P}}{L/\mathbb{Q}}\right) |_K$ is nontrivial, so we have that $\left( \frac{\mathcal{P}}{L/\mathbb{Q}}\right)=\sigma$ or $\left( \frac{\mathcal{P}}{L/\mathbb{Q}}\right)=\tau \sigma $ with $\tau \in C(D)^2$. If $\left( \frac{\mathcal{P}}{L/\mathbb{Q}}\right)=\sigma$, then the class polynomial $H_D(X)$ mod $p$ has roots in $\mathbb{F}_p$. If $\left( \frac{\mathcal{P}}{L/\mathbb{Q}}\right)=\delta^2 \sigma $ with $\delta \in C(D)$, then $\left( \frac{\delta^{-1} \mathcal{P}}{L/\mathbb{Q}}\right)=\delta^{-1} (\delta^2 \sigma ) \delta=\sigma$, which means that the class polynomial $H_D(X)$ mod $p$ has roots in $\mathbb{F}_p$. In conclusion, if $p$ splits completely in $E$, then the class polynomial $H_D(X)$ mod $p$ has roots in $\mathbb{F}_p$.
\end{proof}
\begin{remark}
Suppose $|D|< \frac{4}{\sqrt{3}}\sqrt{p}$ and $\left( \frac{D}{p} \right )=-1$. If the class number $h(D)$ is odd, then there exists exactly one root of $H_D(X)$ mod $p$ in $\mathbb{F}_p$.
\end{remark}

\subsubsection{The Common Roots of Different Class Polynomials}
For an imaginary quadratic order $O_i$ with discriminant $D_i$, we denote $F_i$ its genus field and $E_i=F_i \cap \mathbb{R}$. In this subsection, we assume that $p$ is inert in $\mathbb{Q}(\sqrt{D_i})$ and $p$ splits completely in $E_i$.\par
Let $D_1$ and $D_2$ be two different negative discriminants. If $\mathbb{Q}(\sqrt{D_1})=\mathbb{Q}(\sqrt{D_2})$, then every prime factor $p$ of the resultant of $H_{D_1}(X)$ and $H_{D_2}(X)$ satisfies $p^2\le D_1D_2$ by Theorem \ref{t3}. We assume $|D_i| < \frac{4}{\sqrt{3}}\sqrt{p}$, so $D_1D_2 < \frac{16}{3}p$. Let $p>5$ be a prime integer, then $H_{D_1}(X)$ mod $p$ and $H_{D_2}(X)$ mod $p$ in $\mathbb{F}_p$ have no roots in common. We assume $\mathbb{Q}(\sqrt{D_1})\neq \mathbb{Q}(\sqrt{D_2})$ in the following. \par
Let $D_1$, $D_2$ be two distinct negative discriminants and write
$$J(D_1,D_2)=\prod_{\substack{[\tau_1],[\tau_2] \\ \text{disc}(\tau_i)=D_i}}(j(\tau_1)-j(\tau_2)),$$
where $[\tau_i]$ runs over all elements of the upper half-plane with discriminant $D_i$ modulo $\text{SL}_2(\mathbb{Z})$. Let $w_i$ denote the number of roots of unity in the imaginary quadratic order of discriminant $D_i$. Let $f_i$ denote the conductor of $D_i$. Gross and Zagier \cite{MR772491} studied the prime factorizations of $J(D_1,D_2)$ where $D_1$ and $D_2$ are two fundamental discriminants which are relatively prime. Recently, Lauter and Viray \cite{MR3431591} studied the prime factorizations of $J(D_1,D_2)$ for arbitrary $D_1$ and $D_2$. We need two theorems in \cite{MR3431591}.

\begin{theorem}
Let $D_1$, $D_2$ be any two distinct discriminants. Then there exists a function $F$ that takes non-negative integers of the form $\frac{D_1 D_2 - x^2}{4}$ to (possibly fractional) prime powers. This function satisfies
$$J(D_1, D_2)^{\frac{8}{w_1 w_2}} = \pm \prod_{\substack{x^2 \le D_1D_2 \\ x^2 \equiv D_1 D_2 \pmod 4}} F \left (\frac{D_1 D_2 - x^2}{4} \right ),$$
Moreover, $F(m) = 1$ unless either (1) $m = 0$ and $D_2 = D_1 \ell^{2k}$ for some prime $\ell$ or (2) the Hilbert symbol $(D_1, -m)_{\ell} = -1$ at a unique finite prime $\ell$ and this prime divides $m$. In both of these cases, $F(m)$ is a (possibly fractional) power of $\ell$.
\end{theorem}

\begin{theorem}
Let $m$ be a non-negative integer of the form $\frac{D_1D_2-x^2}{4}$ and $\ell$ a fixed prime that is coprime to $f_1$. If $m > 0$ and either $\ell > 2$ or $2$ does not ramify in both $\mathbb{Q}(\sqrt{D_1})$ and $\mathbb{Q}(\sqrt{D_2})$, then $v_{\ell}(F(m))$ can be expressed as a weighted sum of the number of certain invertible integral ideals in $O_{D_1}$ of norm $m/\ell^r$ for $r > 0$. Moreover, if $m$ is coprime to the conductor of $D_1$, then $v_{\ell}(F(m))$ is an integer and the weights are easily computed and constant; more precisely, we have
$$v_{\ell}(F(m))= \left \{
\begin{array}{ll}
  \frac{1}{e}\rho(m)\sum_{r \geq 1} \mathfrak{A}(m/\ell^r) & \text{if} \ \ell \nmid f_2 , \\
  \rho(m)\mathfrak{A}(m/\ell^{1+v(f_2)}) & \text{if} \ \ell \mid f_2 ,
\end{array}
\right. $$
where $e$ is the ramification degree of $\ell$ in $\mathbb{Q}(\sqrt{D_1})$ and
$$\rho(m)=\left \{
\begin{array}{ll}
  0 & \text{if} \ (D_1,-m)_p=-1 \ \text{for} \ p\mid D_1, p\nmid f_1 \ell, \\
  2^{\# \{p|(m,D_1):p\nmid f_2 \ \text{or} \ p=\ell \} } & \text{otherwise},
\end{array}
\right. $$
$$\mathfrak{A}(N)=\# \left \{
\begin{array}{cl}
   & N(\mathfrak{b})=N, \mathfrak{b} \ \text{invertible,} \\
  \mathfrak{b} \subseteq O_{D_1} & p\nmid \mathfrak{b} \ \text{for all} \ p\mid(N,f_2),p\nmid \ell D_1 \\
   & \mathfrak{p}^3 \nmid \mathfrak{b} \ \text{for all} \ \mathfrak{p}\mid p \mid (N,f_2,D_1), p\neq \ell
\end{array}
\right \}.$$
If $m=0$, then either $v_{\ell}(F(0))=0$ or $D_2=D_1\ell^{2k}$ and
$$v_{\ell}(F(0))=\frac{2}{w_1}\cdot \#\text{Pic}(O_{D_1}).$$
\end{theorem}

We assume $|D_1|, |D_2| < \frac{4}{\sqrt{3}}\sqrt{p}$ and $\mathbb{Q}(\sqrt{D_1})\neq \mathbb{Q}(\sqrt{D_2})$, so $p \mid \frac{D_1D_2-x^2}{4}$ for an integer $x$ if and only if $4p=D_1D_2-x^2$. We will compute $v_{\ell}(F(m))$ for $m= \ell =p$.\par
We assume that $p$ is inert in $\mathbb{Q}(\sqrt{D_1})$, so $e=1$. By definition, $\mathfrak{A}(p/p^{r})$ is $1$ (resp. $0$) if $r =1$ (resp. $r >1$), so $\sum_{r \ge 1} \mathfrak{A}(p/p^r)=1$. The cardinality of $\{q | (p,D_1):q\nmid f_2 \ \text{or} \ q=p \}$ is $1$. Moreover, we have that $(D_1,-p)_q=-1$ at a unique finite prime $q=p$ if $p$ splits completely in $E_1$. We illustrate this when $D_1$ is odd, and other cases can be proved similarly.\par
First, we have $(D_1,-p)_p=\left ( \frac{D_1}{p} \right )=-1$. If $q \nmid pD_1$ or $p \mid f_1$, then $(D_1,-p)_q=1$ obviously. For an odd prime $q | D_1$ and $q \nmid f_1p$, we can write $D_1=q^kD_1'$ where $q \nmid D_1'$. If $2\mid k$, then $(D_1, -p)_q=1$. If $2\nmid k$, we have that
$$(D_1, -p)_q=\left(\frac{-p}{q} \right)^k=\left(\frac{-p}{q} \right).$$
If $q \equiv 1 \pmod 4$, then $\left(\frac{-p}{q} \right)=\left(\frac{q}{p} \right)=1$. If $q \equiv 3 \pmod 4$, then $\left(\frac{-p}{q} \right)=-\left(\frac{-q}{p} \right)=1$ since $\left(\frac{D_1}{p} \right)=-1$ and $\left( \frac{D_1/-q}{p} \right)=1$. We have $\rho(p) =2$ and $v_p(F(p))=2$.
\begin{theorem}
  Let $D_1$ and $D_2$ be two distinct discriminants satisfying $\sqrt{3p} < |D_1| < |D_2| < \frac{4}{\sqrt{3}} \sqrt{p}$ and $\mathbb{Q}(\sqrt{D_1}) \neq \mathbb{Q}(\sqrt{D_2})$. Assume that the prime integer $p>5$ is inert in $\mathbb{Q}(\sqrt{D_1})$ and $\mathbb{Q}(\sqrt{D_2})$. If $p$ splits completely in $E_1$ and $E_2$, then the class polynomials $H_{D_1}(X)$ mod $p$ and $H_{D_2}(X)$ mod $p$ have one common root in $\mathbb{F}_p$ if and only if there exists an integer $x$ such that $D_1D_2-x^2=4p$.
\end{theorem}
\begin{proof}
  If $p \ge 7$, then $|D_2|> |D_1| >4$ and $w_1=w_2=2$. If there exists an integer $x$ such that $D_1D_2-x^2=4p$, then $F(p)=p^2$ and $v_p(J(D_1,D_2))=1$. There exists one common root of $H_{D_1}(X)$ mod $p$ and $H_{D_2}(X)$ mod $p$ and the common root must be in $\mathbb{F}_p$. Otherwise, if $\alpha \in \mathbb{F}_{p^2} \setminus \mathbb{F}_p$ is the common root, then $\bar{\alpha}$ is also a common root and $p^2 | J(D_1,D_2)$. \par
  On the contrary, if $H_{D_1}(X)$ mod $p$ and $H_{D_2}(X)$ mod $p$ have one common root in $\mathbb{F}_p$, then $p | J(D_1,D_2)$ and there exists an integer $x$ such that $D_1D_2-x^2=4p$.
\end{proof}
\begin{remark}
Kaneko proved in \cite{MR1040429} that each prime factor $p$ of the resultant of $H_{D_1}(X)$ and $H_{D_2}(X)$ divides a positive integer of the form $(D_1D_2-x^2)/4$. If $\sqrt{3p} < |D_1| < |D_2| < \frac{4}{\sqrt{3}} \sqrt{p}$ and $4p=D_1D_2-x^2$, Theorem 8 implies that $H_{D_1}(X)$ mod $p$ and $H_{D_2}(X)$ mod $p$ have a common root in $\mathbb{F}_p$.
\end{remark}
For a discriminant $D_1$ satisfying $\sqrt{3p} < |D_1| < \frac{4}{\sqrt{3}} \sqrt{p}$, if there exists a discriminant $D_2$ with $\sqrt{3p} < |D_1| < |D_2| < \frac{4}{\sqrt{3}} \sqrt{p}$ and $D_1D_2 - x^2=4p$ where $x$ is a positive integer, then the form $f(X,Y)=-D_1X^2-2xXY-D_2Y^2$ is a positive definite quadratic form with discriminant $-16p$. We claim that $f(X,Y)$ is reduced.\par
We just need to prove $2x < -D_1$. If not, $2x \ge -D_1$ implies that $16p = 4D_1D_2-4x^2 < 8x \cdot \frac{4\sqrt{p}}{\sqrt{3}}-4x^2$. This inequality holds if and only if $\frac{2\sqrt{p}}{\sqrt{3}} < x < \frac{6\sqrt{p}}{\sqrt{3}}$, which contradicts $x^2=D_1D_2-4p<\frac{4}{3}p$. We do not require that $f(X,Y)$ is primitive.\par
In fact, we have the following theorem.
\begin{theorem}
Let $D_1$, $D_2$ and $D_3$ be three distinct discriminants satisfying $|D_1| < |D_2| < |D_3| < \frac{4}{\sqrt{3}} \sqrt{p}$. Assume that the prime integer $p>5$ is inert in $\mathbb{Q}(\sqrt{D_i})$. If $p$ splits completely in $E_i$, then the class polynomials $H_{D_1}(X)$, $H_{D_2}(X)$ and $H_{D_3}(X)$ have no common root in $\mathbb{F}_p$.
\end{theorem}
\begin{proof}
We assume that $\mathbb{Q}(\sqrt{D_1})$, $\mathbb{Q}(\sqrt{D_2})$ and $\mathbb{Q}(\sqrt{D_3})$ are not equal to each other.
Let $|D_2|$ be the smallest value such that $H_{D_1}(X)$ mod $p$ and $H_{D_2}(X)$ mod $p$ have a common root $j$ in $\mathbb{F}_p$. We must have $|D_1|>\sqrt{3p}$ since $D_1D_2>4p$. As we know, $f(X,Y)=-D_1X^2-2xXY-D_2Y^2$ is a reduced positive definite quadratic form with discriminant $-16p$. On the other hand, we have that $O_i=\mathbb{Z}[\frac{-D_i+\sqrt{D_i}}{2}]$ can be embedded into the endomorphism ring $\mathcal{O}$ of $E(j)$ for $i=1,2$. We assume
$$\mathcal{O}=\mathcal{O}(q,r)= \mathbb{Z} + \mathbb{Z}\frac{1+\beta}{2} + \mathbb{Z} \frac{\alpha(1+\beta)}{2} + \mathbb{Z}\frac{(r+\alpha)\beta}{q},$$
where $q \equiv 3 \ \text{mod} \ 8$, $\left ( \frac{p}{q} \right )=-1$, $\alpha^2=-p$, $\beta^2=-q$ and $\alpha\beta=-\beta\alpha$. If $\mathcal{O}=\mathcal{O}'(q,r')$, the proof is similar.\par
There exist $\gamma_i \in \mathcal{O}$ such that $\text{Trd}(\gamma_i)=-D_i$ and $\text{Nrd}(\gamma_i)=\frac{D_i^2-D_i}{4}$ for $i=1,2$. Let $\gamma_i= w_i + x_i \frac{1+\beta}{2} + y_i \frac{\alpha(1+\beta)}{2} + z_i \frac{(r+\alpha)\beta}{q}$, then we have the following Diophantine equations:
$$ 2w_i + x_i =-D_i$$ and $$(w_i + \frac{x_i}{2})^2 + \frac{p}{4}y_i^2 + q(\frac{x_i}{2}+\frac{z_ir}{q})^2 + pq(\frac{y_i}{2}+ \frac{z_i}{q})^2 = \frac{D_i^2-D_i}{4}.$$
We must have $y_i=0$ since $|D_i| < p$ and $$qx_i^2 + 4rx_iz_i + \frac{4r^2+4p}{q}z_i^2=-D_i.$$
Let $g(X,Y)=qX^2+4rXY+\frac{4r^2+4p}{q}Y^2$. We know that $g(x,y) \equiv 0$ or $3 \pmod 4$ for any $x,y \in \mathbb{Z}$.\par
We assume that the first (resp. second) successive minimum of $g(X,Y)$ is $\sqrt{-D_1}$ (resp. $\sqrt{-D_2}$) which coincides with $f(X,Y)$. By Proposition 5.7.3 and Theorem 5.7.6 of \cite{MR2300780}, $g(X,Y)$ can be reduced to $f(X,Y)$ or $f(X,-Y)$. It is easy to see that the minimal value of the function $f(X,Y)$ is $f(0,0)=0$. If we further assume $X\neq 0$ and $Y\neq 0$, then $f(X,Y) \ge f(1,1)$. Moreover, we have
$$\begin{aligned}
f(1,1)= -D_1-2x-D_2>&2\sqrt{D_1D_2}-2x=2\sqrt{4p+x^2}-2x \\
=&\frac{8p}{\sqrt{4p+x^2}+x} \\
>&\frac{8p}{\sqrt{4p+\frac{4p}{3}}+\frac{2\sqrt{p}}{\sqrt{3}}}\\
=&\frac{4}{\sqrt{3}}\sqrt{p},
\end{aligned}$$
where the second inequality holds since $0<x<\frac{2}{\sqrt{3}}\sqrt{p}$. Suppose that there exists an imaginary quadratic order $O_3$ which can be embedded in $\mathcal{O}(q,r)$, then the discriminant $D_3$ satisfies $|D_3|>\frac{4}{\sqrt{3}}\sqrt{p}$.\par
If $\sqrt{-D_1}$ (resp. $\sqrt{-D_2}$) is not the first (resp. second) successive minimum of $g(X,Y)$, then $|D_2|>\frac{4}{\sqrt{3}}\sqrt{p}$ and this theorem holds. \par
\end{proof}

\section{Computing the class number of $\mathbb{Q}(\sqrt{-p})$}
It is time to present our algorithm (Algorithm 2) to compute the class number of $\mathbb{Q}(\sqrt{-p})$.
\begin{algorithm}[H]
\caption{ Computing the class number of $\mathbb{Q}(\sqrt{-p})$.}
\label{alg:Framwork}
\begin{algorithmic}[1]
\Require
A prime $p>5$;
\Ensure
The class number of $\mathbb{Q}(\sqrt{-p})$;
\State Let $s_p=0$ and $T=\varnothing$;
\label{code:fram:let}
\State For every negative discriminant $D\equiv 0$ or $1 \pmod4$ satisfying $|D| < \frac{4}{\sqrt{3}}\sqrt{p}$, if the Legendre symbol $\left (\frac{D}{p} \right )=-1$, do the following two steps;
\label{code:fram:for}
\State Factor $D$ and compute the generators of the field $E$;
\label{code:fram:if}
\State If $p$ splits completely in $E$, then $s_p=s_p+2^{\mu-1}$; Furthermore, if $|D|>\sqrt{3p}$, then add $D$ to the set $T$;
\label{code:fram:split}
\State Calculate the number $t$ of pairings $(D_1,D_2)$ with $D_1,D_2 \in T$ where $D_1D_2-4p$ is a square;
\State $ h= \left\{
\begin{array}{ll}
\frac{1}{2} (s_{p}-t) &{\text{if} \ p\equiv 3 \pmod{8} ,} \\
s_{p}-t &{\text{if} \ p\equiv 7 \pmod{8} ,}\\
2(s_{p}-t) &{\text{if} \ p\equiv 1 \pmod{4}; }
\end{array}
\right.$
\label{code:fram:equation} \\
\Return $h$.
\end{algorithmic}
\end{algorithm}
The correctness of Algorithm 2 is guaranteed by Theorem 5, 8 and 9.
We now analyze the complexity of Algorithm 2. There are about $\frac{2}{\sqrt{3}}\sqrt{p}$ many $D$ in step 2. We just need to compute less than $\log p$ many Legendre symbols in step 4 which can be calculated in polynomial time of $\log p$.\par
Factoring large integers is difficult. Hittmeir \cite{MR3834692} presented a deterministic algorithm that provably computes the prime factorisation of a positive integer $N$ in $N^{2/9+\epsilon}$ bit operations. Recently, the exponent for deterministic factorization was improved to $1/5$ by Harvey in \cite{harvey2020exponent}.
What's more, Coppersmith algorithm uses lattice basis reduction techniques, and it has the advantage of requiring little memory space. There are several probabilistic algorithms for factoring integers. The quadratic sieve\cite{MR825590}, the number field sieve\cite{MR1321216} and the elliptic curve method \cite{MR916721} are subexponential algorithms for factoring integers.\par
In step 5, it is easy to determine whether $D_1D_2-4p$ is a square or not. There are at most ($\frac{\sqrt{3p}}{3} \times \frac{\sqrt{3p}}{3}$) many pairings $(D_1,D_2)$ in $T$. The time complexity of step 5 is $O(p)$, and the constant is small by further analysis. In general, the time complexity of Algorithm 2 is $O(p)$.\par
As we can see, the step 5 occupies the most time in Algorithm 2. If $D_1D_2-4p$ is a square, then the equation $x^2 \equiv -4p \pmod {|D_2|}$ is solvable. Moreover, if the solution $x$ satisfies $\sqrt{3p} < \frac{x^2+4p}{|D_2|}  < |D_2|$ and $\frac{x^2+4p}{D_2}$ is a negative discriminant, then we get a pairing $(D_1=\frac{x^2+4p}{|D_2|},D_2)$ such that $D_1D_2-4p$ is a square. Moreover, we have Algorithm 3.\par
The correctness of Algorithm 3 is obvious. We will analyze the time complexity of Algorithm 3.\par
We will discuss the time complexity of step 5. Assume $|D|=p_1^{a_1}\cdots p_t^{a_t}$ with $a_i>0$, we have the following equations:
$$\left \{
\begin{array}{c}
x^2 \equiv -4p \pmod {p_1^{a_1}}; \\
\cdots\\
x^2 \equiv -4p \pmod {p_t^{a_t} }.
\end{array}
\right.$$\par
If $p_i>2$, the solutions of $x^2 \equiv -4p \pmod {p_i^{a_i}}$ can be derived from the solutions of $x^2 \equiv -4p \pmod {p_1}$ by Theorem 9.3 in \cite[Chapter 2]{MR665428}. Moreover, one can solve the equation $x^2 \equiv -4p \pmod {p_1}$ in polynomial time if we assume the Generalized Riemann Hypothesis (GRH). If $p_i=2$, then $a_i \ge 2$ and we can solve the equation $x^2 \equiv -p \pmod {2^{a_i-2}}$ by Theorem 5.1 in \cite[Chapter 3]{MR665428}. By this way, it is easy to get the solutions of $x^2 \equiv -4p \pmod {2^{a_i}}$, and we omit the details. If we do not assume the GRH, Bourgain et al.\cite{MR3378857} researched the polynomial factorization problem when the polynomial $f$ fully splits over $\mathbb{F}_p$. They proved the following theorem.
\begin{theorem}
  There is a deterministic algorithm that, given a squarefree polynomial $f \in \mathbb{F}_p[X]$ of degree $n$ that fully splits over $\mathbb{F}_p$, finds in time $(n+p^{1/2})p^{\epsilon}$ a root of $f$.
\end{theorem}
Next, we can get $x \pmod{|D|}$ by Chinese Remainder Theorem. Generally, there are about $2^t$ many solutions where $t$ is the number of distinct prime factors of $|D|$. As we know (see \cite[Chapter 5.10]{MR665428}), $t\sim \log\log(|D|)$ for almost all $|D|$, so there are $O(\log |D|)$ many solutions of $x^2 \equiv -4p \pmod{|D|}$ for almost all $|D|$. As we discussed above, step 5 can be done in polynomial time, and factoring $D$ in step 3 occupies the most time in Algorithm 3.
\begin{algorithm}[H]
\caption{ Computing the class number of $\mathbb{Q}(\sqrt{-p})$.}
\label{alg:Framwork}
\begin{algorithmic}[1]
\Require
A prime $p$;
\Ensure
The class number of $\mathbb{Q}(\sqrt{-p})$;
\State Let $s_p=0$ and $t=0$;
\label{code:fram:let}
\State For every negative discriminant $D\equiv 0$ or $1 \pmod4$ satisfying $|D| < \frac{4}{\sqrt{3}}\sqrt{p}$, if the Legendre symbol $\left (\frac{D}{p} \right )=-1$, do the following three steps;
\label{code:fram:for}
\State Factor $D$ and compute the generators of the field $E$;
\label{code:fram:if}
\State If $p$ splits completely in $E$, then $s_p=s_p+2^{\mu-1}$; Furthermore, if $|D|>2\sqrt{p}$, then do the next step;
\label{code:fram:split}
\State Solve the equation $x^2 \equiv -4p \pmod {|D|}$ and calculate the number $t_D$ of solutions satisfying $\sqrt{3p} < \frac{x^2+4p}{|D|}  < |D|$ and $\frac{x^2+4p}{D} \equiv 0,1 \pmod 4$. Let $t=t+t_D$;
\State $ h= \left\{
\begin{array}{ll}
\frac{1}{2} (s_{p}-t) &{\text{if} \ p\equiv 3 \pmod{8} ,} \\
s_{p}-t &{\text{if} \ p\equiv 7 \pmod{8} ,}\\
2(s_{p}-t) &{\text{if} \ p\equiv 1 \pmod{4}; }
\end{array}
\right.$
\label{code:fram:equation} \\
\Return $h$.
\end{algorithmic}
\end{algorithm}

\begin{theorem}
The time complexity of Algorithm 3 is $O(p^{3/4+\epsilon})$. Moreover, if we factor $D$ by probabilistic algorithms and assume the Generalized Riemann Hypothesis (GRH), the time complexity can be reduced to $\tilde{O}(p^{1/2})$.
\end{theorem}
\begin{proof}
  As discussed above, for every $D$ with $\left ( \frac{D}{p} \right )=-1$, we need factor $|D|$ in step 3. Moreover, for $|D|>2\sqrt{p}$, we also need solve the equation $x^2 \equiv -4p \pmod {|D|}$, which can be done in polynomial time under the Generalized Riemann Hypothesis (GRH)or in time $O(p_i^{1/2+\epsilon})(O(p^{1/4+\epsilon}))$ by Theorem 10. Other steps can be computed in polynomial time. There are about $\frac{2}{\sqrt{3}}\sqrt{p}$ many $D$, so this theorem holds.
\end{proof}
We present in the following table the results of applying Algorithm 3 to compute class numbers of $\mathbb{Q}(\sqrt{-p})$ with various size discriminants. The computations were all carried out on a PC microcomputer with Intel(R) Core(TM) i5-10210U by PARI/GP. Here $h$ is the class number of the quadratic field $\mathbb{Q}(\sqrt{-p})$, and $T$ is the time required to compute the class number of $\mathbb{Q}(\sqrt{-p})$ by Algorithm 3. All these $T$ are expressed in second.
\begin{table}[H]
  \centering
  \setlength{\tabcolsep}{5mm}{
  \begin{tabular}{l r r }
    \toprule
    $p$ & $h$ &  $T$ \\
    \midrule
    $10^{11}+283$ & $88847$ &  $14$ \\
    $10^{12}+547$ & $240171$ &  $50$ \\
    $10^{13}+99$ & $670135$ &  $182$ \\
    $10^{14}+99$ & $1981515$ &  $694$ \\
    $10^{15}+9867$ & $4921593$ &  $2605$\\
    \bottomrule
  \end{tabular}}
\end{table}
The memory space of Algorithm 3 can be neglected. For similar-sized $p$'s, the running time changes as the class number changes. We use the deterministic factorization algorithms in the program. The data in the table illustrate the correctness and effectiveness of Algorithm 3.

\section{Conclusion}
For a prime $p$, let $D_1$, $D_2$ be any two negative discriminants. Assume that $D_1$ and $D_2$ are not quadratic residues in $\mathbb{F}_p$ with $|D_1|< |D_2| < \frac{4\sqrt{p}}{3}$, then $H_{D_1}(X)$ and $H_{D_2}(X)$ have a common root in $\mathbb{F}_p$ if $D_1D_2-4p$ is a square. The case which $H_{D_1}(X)$ and $H_{D_2}(X)$ have a common root in $\mathbb{F}_{p^2} \setminus \mathbb{F}_p$ is also interesting, and it may be more difficult.\par
The correctness of our algorithms, Algorithm 2 and Algorithm 3, are guaranteed by Theorem 5, Theorem 8 and Theorem 9. In fact, we use neither the supersingular $j$-invariants nor the class polynomials in the calculation process.

\section{Acknowledge}
The authors would like to thank Shparlinski for suggestions on the complexity of Algorithm 3 and Jianing Li for comments on the proof of Theorem 5.

\bibliographystyle{amsplain}
\bibliography{reference}

\end{document}